\documentclass{amsart}

\usepackage[utf8]{inputenc}

\usepackage[english]{babel}

\usepackage{amsmath,amssymb,amsthm,stackengine}

\usepackage{esint}
\usepackage{enumitem}

\usepackage{babelbib}

\usepackage{hyperref}

\usepackage[T1]{fontenc}

\usepackage{lmodern}

\theoremstyle{definition}

\newtheorem{lemma}{Lemma}
\newtheorem{thm}{Theorem}
\newtheorem{rem}{Remark}
\newtheorem{prop}{Proposition}
\newtheorem{cor}{Corollary}

\author{Glal Bacho}
\thanks{Glal Bacho (corresponding author): Lehrstuhl für Angewandte Analysis, RWTH Aachen University, Kreuzherrenstr. 2, 52062 Aachen, Germany. Email: \texttt{bacho@math1.rwth-aachen.de}. Phone: +49 241 80-94584.}

\author{Christof Melcher}
\thanks{Christof Melcher: Lehrstuhl für Angewandte Analysis, RWTH Aachen University, Kreuzherrenstr. 2, 52062 Aachen, Germany. Email: \texttt{melcher@rwth-aachen.de}. Phone: +49 241 80-94585.}

\title[Hopfions in frustrated magnets]{Existence and large-field asymptotics of Hopfions in a continuum model of frustrated magnets}

\newcommand{\bs}[1]{\boldsymbol{#1}}
\newcommand{\e}[1]{\bs{\hat{e}}_{#1}}
\renewcommand{\d}[1]{\mathrm{d}#1}

\newcommand{\R}{\mathbb{R}}
\newcommand{\C}{\mathbb{C}}
\newcommand{\Z}{\mathbb{Z}}
\newcommand{\N}{\mathbb{N}}
\renewcommand{\S}{\mathbb{S}}
\renewcommand{\Re}{\operatorname{Re}}
\renewcommand{\Im}{\operatorname{Im}}

\newcommand{\E}{\mathcal{E}}
\newcommand{\tE}{\widetilde{\E}}

\newcommand{\m}{\boldsymbol{m}}

\newcommand{\loc}{\mathrm{loc}}

\newcommand{\F}{\boldsymbol{F}}
\newcommand{\A}{\boldsymbol{A}}

\newcommand{\eps}{\varepsilon}

\newcommand{\ein}{\boldsymbol{\hat{e}}}

\newcommand{\qand}{\quad \text{and} \quad}

\begin{document}

\begin{abstract}
We study a continuum model of frustrated magnetism that supports particle-like field configurations with nontrivial Hopf index. The energy functional combines competing second and first order exchange interactions with a stabilizing external magnetic field and gives rise to a variational problem for maps from three-dimensional space into the two-sphere. We prove the existence of energy minimizers in infinitely many nontrivial homotopy classes. We further establish the existence of a unit-charge Hopfion in the regime of strong external fields. Finally, we analyze the corresponding large-field limit, identify the natural concentration scale, and prove strong convergence of suitably rescaled minimizers to a minimizer of a limiting variational problem.
\end{abstract}

\maketitle

%\tableofcontents

\section{Introduction}
We consider maps $\m:\mathbb{R}^3\to\S^2$ that approach a constant at infinity. Compactifying $\mathbb{R}^3\cup\{\infty\}\simeq\S^3$ then identifies such maps with maps $\m:\S^3\to\S^2$, whose homotopy classes are classified by $\pi_3(\S^2)\simeq\mathbb{Z}$. Thus, each map carries an integer-valued topological invariant $Q[\m]$, known as the Hopf index, which can be interpreted as the linking number of two generic fibers of $\m$. Maps with nonzero Hopf index arise in a range of physical settings, from topological hydrodynamics and electromagnetism to nonlinear field theories, where the corresponding knotted or linked structures may appear as stationary or dynamical configurations. Such configurations are commonly referred to as Hopfions. A prototypical continuum model supporting such Hopfions is the Skyrme-Faddeev model, obtained by modifying the $\mathcal{O}(3)$ sigma model \cite{faddeevquantization,faddeevtoroidal}
\begin{align*}
E_F[\m]
= \int_{\R^3} |\nabla \m|^2 \;+\; |\bs{F}|^2 \d x
\end{align*}
where $\bs{F}=\nabla\times \A$ is the curvature induced by the pull-back via $\m$, i.e.,
\begin{align*}
\F \;=\; \frac12\varepsilon_{ijk} \big(\m\cdot \partial_i \m\times \partial_j \m\big)\e{k},
\end{align*}
that is directly related to the Hopf index via the Whitehead integral representation
\begin{align*}
Q[\m] \;=\; \frac{1}{16\pi^2}\int_{\R^3} \F\cdot \A\, \d x,
\end{align*}
which exhibits $Q$ as a gauge-invariant helicity-type quantity \cite{whitehead}. A key feature of this coupling is the existence of topological lower bounds of Vakulenko--Kapitanskii type,
$E_F[\m] \gtrsim |Q[\m]|^{\frac34}$, which ensure the well-definedness of the Hopf index in the relevant energy class and play a decisive role in the existence of local minima within the corresponding topological sectors of configuration space. In particular, the curvature term counteracts the tendency of nontrivial configurations to collapse to arbitrarily small scales, thereby stabilizing a finite characteristic length scale. Lin and Yang proved the existence of minimizing configurations for infinitely many Hopf charges \cite{LinYang2004} and subsequently obtained existence in the unit-charge classes $Q=\pm1$ using their so-called substantial inequality \cite{LinYang2007}. On closed three-manifolds, Auckly and Kapitanski developed a corresponding variational existence in every homotopy class, based on the Pontrjagin classification of maps into $\S^2$ \cite{Auckly2005GaugeExistence}. \\

A different route to Hopfions emerges from continuum limits of classical Heisenberg spin systems with competing (frustrated) interactions in a stabilizing uniaxial Zeeman field. 
An effective micromagnetic model is given by
\begin{align*}
\E_H[\m]
\;=\; \frac12 \int_{\R^3}  |\Delta \m|^2 \;-\; |\nabla \m|^2 \;+\; H | \m - \ein_3|^2 \d x,
\end{align*}
with $H>0$ the (rescaled) field strength \cite{lin2016ginzburg}. This model couples a stabilizing Laplacian term to a destabilizing gradient energy and a Zeeman term. We work in the regime $H>\tfrac{1}{4}$, which ensures a coercive localization of the energy around the uniform ground state and supports a topological lower bound, linking $\E_H$ to the Hopf index. A closely related situation arises in the two-dimensional Skyrmion case, where the corresponding topological solitons have been shown to be supported \cite{lin2016ginzburg}. Their mathematically rigorous existence, however, has only recently been established \cite{Li2026Skyrmions}. In three dimensions, compelling numerical evidence for the existence of minimizing Hopfions in frustrated magnets was provided in the influential work of \cite{Sutcliffe2017Frustratedknots}, which demonstrated the persistence of knotted topological structures in this setting. More broadly, Hopfions have been shown to arise in material-realistic models of magnetic solids with competing exchange interactions, providing a connection between the mathematical theory of topological solitons and the search for realizations in magnetic materials \cite{RybakovKiselevHopfions}.
Corresponding dynamical equations have also been studied analytically. In particular, the frustrated-magnets model has been incorporated into Landau--Lifshitz--Gilbert dynamics coupled to Vlasov--Maxwell equations through emergent electromagnetic fields, as well as into stochastic Landau--Lifshitz--Gilbert dynamics driven by fluctuating spin currents \cite{DoresicMelcher,GoldysJiaoMelcher}, establishing the persistence of topological patterns under transport noise. Nevertheless, the variational existence of energy-minimizing configurations in nontrivial homotopy classes has remained open. This provides the natural starting point for the variational existence problem considered here.

Using integration by parts, the energy density may be recentered at the uniform background state $\e 3$
\begin{align}
    \label{eq:energyDensity}
    e(\m)
    =
    \frac12\left(
    \Big|\Delta \m+\tfrac{1}{2}(\m-\e{3})\Big|^2
    +
    \Big(H-\tfrac{1}{4}\Big)|\m-\e{3}|^2
    \right)
\end{align}

For $H>\tfrac{1}{4}$, the quadratic form in $\Delta\m$ and $\m-\e{3}$ is positive definite. Hence
\begin{align}
    \label{eq:densityCoercivity}
    e(\m)
    \gtrsim_H
    |\Delta\m|^2
    +
    |\m-\e{3}|^2.
\end{align}
This reveals the natural configuration space
\begin{align*}
H^2(\R^3;\S^2)\;=\;\Big\{\m:\R^3\to\S^2 \;\Big|\; \m-\e{3}\in H^2(\R^3;\R^3)\Big\},
\end{align*}
on which $Q$ is well defined. In this situation, a Vakulenko-Kapitanskii-type estimate 
\begin{align} \label{eq:Harland}
\E_H[\m] \;\ge\; C(H)|Q[\m]|^{\frac34}
\end{align}
due to Harland \cite{harland} with an explicit $H$-dependent constant 
\begin{align*}
C(H) = \frac{2^{\frac{9}{2}} \pi^2}{3^{\frac58}} \frac{\sqrt{1+12H}-2}{\sqrt{\sqrt{1+12H}-1}} 
\end{align*}
holds true. The variational problem is therefore
\begin{align}\label{Ek}
\tag{$E_k$}
I_k(H) \;=\; \inf\big\{\E_H[\m] : \m\in H^2(\R^3;\S^2),\; Q[\m]=k\big\}.
\end{align}
A minimizer $\m_0$ for \eqref{Ek}, when it exists, is a static Hopfion of charge $k \in \Z\setminus\{0\}$. Spatial inversion, $\m^-(x)=\m(-x)$, preserves the energy and reverses the Hopf charge. Consequently,
\begin{align}\label{eq:reflection}
    I_k(H)=I_{-k}(H),
\end{align}
so it suffices to consider positive charges $k>0$, and in particular $k=1$ for unit charge.

At the analytical level, the lack of compactness is addressed by the concentration-compactness principle of Lions \cite{lions}. In the subcritical $H^2$-setting considered here, the relevant obstruction is dichotomy, namely a nontrivial splitting into lower energy levels. Once such splittings are excluded, minimizing sequences are compact up to translations and the corresponding energy level is attained. Following the strategy of Lin and Yang, the sublinear growth of the minimal energy with respect to the Hopf charge can then be exploited to exclude nontrivial splittings for infinitely many such charges \cite{LinYang2004}. 

\begin{thm}\label{thm:largeHunit} $ $ \\ \vspace*{-0.4cm}
\begin{enumerate}[label=(\roman*), ref=(\roman*)]
    \item\label{item:infinitelymany} For $H>\frac14$, $I_k(H)$ is attained infinitely many $k \in \Z$.
    \item\label{item:largeHexistence} There exists $H_0>\tfrac{1}{4}$ such that $I_1(H)$ is attained for every $H\geq H_0$.
\end{enumerate}
\end{thm}

In the large field regime $H \gg 1$ the scaling structure of the energy can be exploited. The minimal energy in a fixed class grows like $H^{\frac14}$ and is governed at leading order by the scale invariant quantity
\begin{align} \label{eq:scaleinvariantenergy}
\|\m-\e{3}\|_{L^2}^{\frac12}
\|\Delta\m\|_{L^2}^{\frac32}.
\end{align}
At leading order, the strict subadditivity condition $I_1(H)<I_\ell(H)+I_{1-\ell}(H)$ for the unit charge can be verified on a single, well-chosen competitor.

The competitor is built from the standard Hopf map, the prototypical unit charge configuration obtained by composing the inverse stereographic projection $\R^3\cup\{\infty\}\simeq\S^3$ with the Hopf fibration $\S^3\to\S^2$. This map is not admissible, since it approaches the background $\e{3}$ only like $|x|^{-1}$, so that $\m-\e{3}\notin L^2(\R^3)$ and the energy is infinite. To restore finite energy without changing the topology, we modify the radial profile to ensure exponential decay toward $\e{3}$ at infinity.

\medskip

A standard scaling analysis of minimizers guarantees a characteristic asymptotic core scale of $H^{-\frac14}$. Using the scale invariant asymptotic energy \eqref{eq:scaleinvariantenergy} we can identify a limiting variational integral
\begin{align*}
    \E^\infty(\m)
    =
    \frac12
    \int_{\R^3}
        |\Delta\m|^2
        +
        |\m-\e{3}|^2
    \; \d{x}
\end{align*}
and recycle the concentration-compactness route to ensure compactness modulo translation in this concentrating scale.

\begin{thm}\label{thm:profile}
For $H\ge H_0$ let $\m_H\in H^2(\R^3;\S^2)$ be a minimizer of $I_1(H)$. Then the following holds.
\begin{enumerate}[label=(\roman*), ref=(\roman*)]
\item\label{item:profileTrivial} As $H\to\infty$,
\begin{align*}
   \| \m_H- \e{3} \|_{H^1(\R^3)} \to 0
    \qquad\text{while}\qquad
    \|\Delta\m_H\|_{L^2(\R^3)}\to\infty .
\end{align*}
In particular $(\m_H-\e{3})_{H\ge H_0}$ is unbounded in $H^2(\R^3;\R^3)$.
\item\label{item:profileLimit} For every sequence $H_j\to\infty$ there exist translations $y_j\in\R^3$, a subsequence (not relabelled) and a map $\bs{u}\in H^2(\R^3;\S^2)$ such that
\begin{align*}
    \bs{u}_j
    :=
    \m_{H_j}\bigl(H_j^{-\frac14}(\,\cdot\,-y_j)\bigr)
    \longrightarrow
    \bs{u}
    \quad\text{strongly in } H^2(\R^3;\R^3)
\end{align*}
(and thus uniformly) with $Q[\bs u]=1$ and
\[
 \E^\infty(\bs{u}) \;=\; \inf\big\{\E^\infty[\m] : \m\in H^2(\R^3;\S^2),\; Q[\m]=1\big\}.
\]
\end{enumerate}
\end{thm}
\section{Vakulenko-Kapitanskii-type estimates}
\subsection{Lower energy bounds} \label{sec:lowerbounds}
For $\lambda > 0$, define the scaled map $\m^\lambda(x) = \m(\lambda x)$ and note that $Q[\m^{\lambda}] = Q[\m]$ due to homotopy invariance of the Hopf index. In contrast, the energy is not invariant under this transformation
\begin{align}\label{eq:scaling}
\E_H[\m^\lambda]
= \frac{\lambda}{2}\|\Delta \m\|_{L^2}^2
- \frac{1}{2\lambda}\|\nabla \m\|_{L^2}^2
+ \frac{H}{2\lambda^3}\|\m-\e{3}\|_{L^2}^2.
\end{align}
The presence of the $|\Delta \m|^2$ term prevents the collapse that typically appears in purely quadratic models, analogously to the stabilizing mechanism in the Skyrme-Faddeev model.
\begin{lemma}\label{lemma:scaling}
Let $\m\in H^2(\R^3;\S^2)$. Then there exists $\theta\in [0,1]$ (depending on $\m$) such that
\begin{align}
    \label{eq:scalingFormula}
  \inf_{\lambda>0}\E_H[\m^{\lambda}]
=
\frac{\sqrt{2}}{3}
\frac{\sqrt{\theta^2+12H}-2\theta}
{\sqrt{\sqrt{\theta^2+12H}-\theta}}
\,
\|\m-\e{3}\|_{L^2}^{\frac12}
\|\Delta\m\|_{L^2}^{\frac32}.
\end{align}
In particular
\begin{align*}
      \frac{\sqrt{2}}{3}\frac{\sqrt{1+12H}-2}{\sqrt{\sqrt{1+12H}-1}}\,\|\m-\e{3}\|_{L^2}^{\frac12}\|\Delta\m\|_{L^2}^{\frac32}\le \inf_{\lambda>0}\E_H[\m^{\lambda}]\le \frac{\sqrt{2}}{3}(12H)^{\frac14}\,\|\m-\e{3}\|_{L^2}^{\frac12}\|\Delta\m\|_{L^2}^{\frac32}.
\end{align*}
\end{lemma}
\begin{proof}
Define for $\lambda>0$, $\m^{\lambda}(x)=\m(\lambda x)$ and as in \eqref{eq:scaling}
\begin{align*}
\E_H[\m^{\lambda}] &=  \frac{\lambda \| \Delta \m \|_{L^2}^2}{2} - \frac{ \|\nabla \m\|_{L^2}^2}{2\lambda} + \frac{H\|\m-\e{3}\|_{L^2}^2}{2\lambda^3}\\
&=\frac12
\|\m-\e{3}\|_{L^2}^{\frac12}
\|\Delta\m\|_{L^2}^{\frac32}
\left(
y-\frac{\theta}{y}+\frac{H}{y^3}
\right),
\end{align*}
where (assuming $\m \neq \e 3$)
\begin{align*}
y
=
\lambda
\frac{\|\Delta\m\|_{L^2}^{\frac12}}
{\|\m-\e{3}\|_{L^2}^{\frac12}},
\qquad
\theta
=
\frac{\|\nabla\m\|_{L^2}^{2}}
{\|\m-\e{3}\|_{L^2}\|\Delta\m\|_{L^2}} \in [0,1].
\end{align*}
The minimum in $y>0$ is attained at
\begin{align*}
y =\sqrt \frac{\sqrt{\theta^2+12H}-\theta}{2}.
\end{align*}
Since $\inf_{\lambda>0} \E_H[\m^\lambda]$ is decreasing in $\theta$, the desired estimates follow.
\end{proof}
\begin{lemma}[Scaling-invariant topological lower bound]\label{lemma:scaling_invariant_topological_bound}
Let $\m\in H^2(\R^3;\S^2)$. Then
\begin{align*}
    \|\m-\e{3}\|_{L^2}^{\frac12}
    \|\Delta\m\|_{L^2}^{\frac32}
    \ge
    2^4 3^{\frac38}\pi^2 |Q[\m]|^{\frac34}.
\end{align*}
\end{lemma}

\begin{proof}
We use the Whitehead representation
\begin{align*}
    Q[\m]
    =
    \frac{1}{16\pi^2}
    \int_{\R^3}
    \A\cdot\F
    \d{x},
    \qquad
    \nabla \times \A=\F,
    \qquad
    \nabla \cdot \A=0.
\end{align*}
By Hölder's inequality,
\begin{align*}
    |Q[\m]|
    \le
    \frac{1}{16\pi^2}
    \|\A\|_{L^6}
    \|\F\|_{L^{\frac65}}.
\end{align*}
In Coulomb gauge and by the sharp Sobolev inequality on $\R^3$ (cf. \cite{Talenti1976BestSobolevConstant}),
\begin{align*}
    \|\A\|_{L^6}
    \le
    \frac{2^{\frac23}}{\sqrt{3}\pi^{\frac23}}
    \|\nabla\A\|_{L^2}
    =
    \frac{2^{\frac23}}{\sqrt{3}\pi^{\frac23}}
    \|\F\|_{L^2}.
\end{align*}
and Hölder interpolation $L^1\cap L^2\subset L^{\frac65}$ yield
\begin{align*}
    |Q[\m]| \le \frac{2^{\frac23}}{16\sqrt{3}\pi^{\frac83}} \|\F\|_{L^1}^{\frac23} \|\F\|_{L^2}^{\frac43}.
\end{align*}
As in the Vakulenko-Kapitanskii-type estimate used by Ward \cite{Ward1999Hopfions}, and with $ \m\cdot\Delta\m=-|\nabla\m|^2$, one has
\begin{align*}
    |\F| \le \frac12|\nabla\m|^2 \qand |\F| \le \frac12|\Delta\m|.
\end{align*}
Hence, we obtain the intermediate estimate
\begin{align}
    \label{eq:intermediateVK}
    |Q[\m]|^{\frac34}
    \le
    \frac{1}{2^4 3^{\frac38}\pi^2}
    \|\nabla\m\|_{L^2}
    \|\Delta\m\|_{L^2}.
\end{align}
Finally, using integration by parts yields
\begin{align*}
|Q[\m]|^{\frac34} \le \frac{1}{2^4 3^{\frac38}\pi^2} \|\m-\e{3}\|_{L^2}^{\frac12} \|\Delta\m\|_{L^2}^{\frac32}.
\end{align*}
\end{proof}

Lemma~\ref{lemma:scaling} and Lemma~\ref{lemma:scaling_invariant_topological_bound} together result in an alternative proof of the topological lower bound in \cite{harland}
\begin{align*}
     \E_H[\m] \;\ge\; \frac{2^{\frac{9}{2}} \pi^2}{3^{\frac58}} \frac{\sqrt{1+12H}-2}{\sqrt{\sqrt{1+12H}-1}} |Q[\m]|^{\frac34}
\end{align*}
for every $\m \in H^2(\R^3;\S^2)$. Harland suggested that a stronger lower bound might result from exploiting all terms in the energy simultaneously. Since our exact scaling optimization already retains the full energy and still yields the same coefficient, such an improvement cannot arise from the scaling balance alone. It would require a genuinely sharper topological estimate or additional structural information.

For our purposes, it is crucial to obtain localized estimates on the topological density. The argument is closely related to the classical Vakulenko-Kapitanskii bound \cite{vakulenkostability}.

\begin{lemma}\label{lemma:loc}
Let $B \subset \R^3$ be measurable and let $\m \in H^2(\R^3;\S^2)$. There exist constants $C_0, C_1 > 0$, independent of $\m$ and $B$, such that
\begin{align*}
\int_B |\F \cdot \A|  \d x  &\le C_0 \E_H[\m] |B|^{\frac13},\\[0.5em]
\int_B |\F \cdot \A|  \d x  &\le C_1 \E_H[\m]^{\frac76} \E_H[\m;B]^{\frac16},
\end{align*}
where $\A\in L^6(\R^3;\R^3)$ is chosen in Coulomb gauge, i.e. $\nabla \cdot \A = 0$.
\end{lemma}
\begin{proof}
Using the same estimates as in the proof of Lemma~\ref{lemma:scaling_invariant_topological_bound}, applied locally on $B$ where appropriate, we obtain
\begin{align*}
\int_B |\F\cdot\A|\d x
\lesssim \|\nabla \m\|_{L^2(B)}^{\frac43}\|\nabla \m \|_{L^4(B)}^{\frac23}\|\nabla \m\|_{L^4(\R^3)}^2.
\end{align*}
Thus 
\begin{align*}
    \|\nabla \m\|_{L^2(B)}^{\frac43} \le c |B|^{\frac13} \| \nabla \m\|_{L^4(B)}^{\frac43} \quad \text{ and } \quad  \|\nabla \m\|_{L^4(\R^3)}^4 \lesssim \E_H[\m]
\end{align*}
yield the first inequality. To obtain the second estimate, we use 
\begin{align*}
\| \nabla \m\|_{L^4(B)}^{\frac23}\lesssim \E_H[\m;B]^{\frac16} \quad \text{ and } \quad \| \nabla \m\|_{L^2(\R^3)}^2 \lesssim \E_H[\m].
\end{align*}
\end{proof}

\subsection{Upper energy bound} \label{sec:upperbound}
%Note that the lower bound in \cite{harland} reads
%\begin{align*}
%    I_k(H) \gtrsim H^{\frac14}|k|^{\frac34}.
%\end{align*}
We derive a matching upper bound for \eqref{eq:Harland}.
\begin{lemma}\label{lemma:energygrowth}
There exists a constant $C>0$ such that for all $k\in\Z$ and $H>\tfrac14$,
\begin{align*}
I_k(H) \le CH^{\frac14}|k|^{\frac34}.
\end{align*}
\end{lemma}
\begin{proof}
We may assume $k>0$. Choose $\ell\in\N$ such that
\begin{align*}
\ell^{2} \le k < (\ell+1)^{2}.
\end{align*}
By \cite{LinYang2004}, there exist maps $\overline{\m},\m_{0}\in H^{2}(\R^{3};\S^{2})$ satisfying
\begin{align*}
\overline{\m}=\e{3} \text{ on } \R^{3}\setminus B_{\sqrt{\ell}}, \quad Q[\overline{\m}]=\ell^{2}, \\
\m_{0}=\e{3} \text{ on } \R^{3}\setminus B_{1},\quad Q[\m_{0}]=1,
\end{align*}
and a constant $C>0$ such that
\begin{align*}
\|\nabla^{2}\overline{\m}\|_{L^{\infty}} \le C, \quad \|\nabla^{2}\m_{0}\|_{L^{\infty}} \le C.
\end{align*}
Define $\m\in H^{2}(\R^{3};\S^{2})$ by setting $\m=\overline{\m}$ on $B_{\sqrt{\ell}}$ and placing $(k-\ell^{2})$ translated copies of $\m_{0}$ at points $x_{1},\dots,x_{k-\ell^{2}}$ with
\begin{align*}
|x_{i}-x_{j}|\gg 3,\quad |x_{i}|\gg\sqrt{\ell}+1.
\end{align*}
Exploiting the quasi-locality of the Hopf index density (in Coulomb gauge), the supports may be chosen sufficiently far apart so that
\begin{align*}
Q[\m] = \ell^{2} + (k-\ell^{2}) = k.
\end{align*}
Lemma~\ref{lemma:scaling}, together with Young's inequality, then yields
\begin{align*}
   I_k(H) \le  \inf_{\lambda >0}\E_H[\m^{\lambda}] \lesssim H^{\frac14}\int_{\R^3} |\nabla^2 \m|^2 + |\m - \e{3}|^2 \d x .
\end{align*}
Using disjointness of supports and the uniform bounds on $\nabla^{2}\overline{\m}$ and $\nabla^{2}\m_{0}$, together with $|B_{\sqrt{\ell}}|=|B_1|  \ell^{\frac32}$ and $0 \le k-\ell^{2}\le 2\ell+1$, we obtain
\begin{align*}
I_k(H) \lesssim H^{\frac14}(\ell^{\frac32} + \ell)\lesssim H^{\frac14} k^{\frac34}.
\end{align*}
\end{proof}

\section{Energy splitting and compactness}\label{sec:dichotomy}
Fix $H>\frac14$ and a minimizing sequence $(\m_n)_{n\in\mathbb N}\subset H^2(\mathbb R^3;\mathbb S^2)$ of \eqref{Ek} with
\begin{align*}
\E_H[\m_n] \to I_k(H) \quad \text{as } n\to\infty,
\end{align*}
and set
\begin{align*}
M := \sup_{n\in\mathbb N} \E_H[\m_n] < \infty .
\end{align*}
Applying the concentration-compactness principle by P.L.Lions (cf. \cite{lions}) to the energy density $e(\m_n)$ yields the following three cases:
\begin{enumerate}[label=(\roman*), ref=(\roman*)]
\item\label{case:vanishing} vanishing:
For every $R>0$ we have that
\begin{align*}
\lim \limits_{n \to \infty} \sup \limits_{y \in \R^3} \int_{B_R(y)} e(\m_n) \d{x} = 0
\end{align*}
\item\label{case:compactness} compactness: There exists a sequence $(y_n)_{n \in \N}$ such that for every $\varepsilon > 0$ we can find a radius $R>0$ with 
\begin{align*}
\int_{\R^3 \setminus B_R(y_n)} e(\m_n) \d{x} \le \varepsilon
\end{align*}
for $n \in \N$ large enough.
\item\label{case:dichotomy} dichotomy: There is a $0 < \lambda < 1$ such that we can find for $\varepsilon > 0 $ a radius $R>0$, a sequence of radii $R_n > 0$, $R_n \to \infty$, and a sequence $(y_n)_{n \in \N}$ such that 
\begin{align*}
&\left| \lambda \E_H[\m_n] - \E_H[\m_n;B_R(y_n)]\right|< \eps \quad \text{and} \\&\left| (1-\lambda)\E_H[\m_n] - \E_H[\m_n;\R^3 \setminus B_{R_n}(y_n)]\right|  < \varepsilon,
\end{align*}
for $n \in \N$ large enough.
\end{enumerate}

\begin{lemma}\label{lemma:vanishing}
Let $(\m_n)_{n \in \N}$ be a minimizing sequence for $I_k(H)$ with $k\in\mathbb Z\setminus\{0\}$.
Then, alternative~\ref{case:vanishing} vanishing cannot occur.
\end{lemma}

\begin{proof}
In the vanishing case, interior elliptic estimates yield
\begin{align*}
    \|\m_n - \e 3\|_{H^2(B_1(y))}^2\lesssim \int_{B_2(y)} |\Delta \m_n|^2 + |\m_n - \e 3|^2 \d x 
\end{align*}
for every $y \in \R^3.$ Moreover, $H^2_{\mathrm{loc}}\hookrightarrow C^0_{\mathrm{loc}}$ implies
\begin{align*}
\|\m_n-\e{3}\|_{L^\infty(\R^3)} =\sup_{y \in \R^3}\|\m_n-\e{3}\|_{L^\infty(B_1(y))}\lesssim \sup_{y \in \R^3}\|\m_n-\e{3}\|_{H^2(B_1(y))}\to 0. 
\end{align*}
Hence $\m_n$ is null-homotopic for large $n\in \N$, contradicting $Q[\m_n]=k\neq0$.
\end{proof}

To handle the dichotomy case, we construct localized competitors for configurations in $H^2(\R^3,\S^2)$. 
\begin{lemma} \label{cut-offconstruction}
There exists a $\varepsilon_0 > 0$ such that if for $0<\varepsilon<\varepsilon_0$, $R \geq 1$ and $\m \in H^2(\R^3;\S^2)$ we have
\begin{align*}
\int_{B_{4R}\setminus B_R} |\m- \e{3}|^2 + |\Delta \m|^2 \d x \le \varepsilon,
\end{align*}
then we can find $\m^{(1)}$ and $\m^{(2)}$ with the following properties:
\begin{align*}
&\m^{(1)} = \m \text{ on } B_{2R}  &&\m^{(2)} = \m \text{ on } \R^3 \setminus B_{2R} \\
&\m^{(1)} = \e{3} \text{ on } \R^3 \setminus B_{4R} &&\m^{(2)} = \e{3} \text{ on }  B_{R} \\
\end{align*}	
and, furthermore, there is a constant $c>0$ with
\begin{align*}
\lVert \m^{(1)} - \e{3} \rVert_{H^2(\R^3 \setminus B_{2R})}^2 + \lVert \m^{(2)} - \e{3} \rVert_{H^2(B_{2R})}^2 \le c \varepsilon.
\end{align*}
\end{lemma}
\begin{proof}
We construct only $\m^{(1)}$. The construction of $\m^{(2)}$ is analogous.
By the interior elliptic estimate on annuli, applied after scaling, we have
\begin{align*}
    \|\m-\e{3}\|_{H^2(B_{3R}\setminus B_{2R})}^2
    \le
    C
    \int_{B_{4R}\setminus B_R}
    |\m-\e{3}|^2+|\Delta\m|^2 \d x ,
\end{align*}
where $C$ is independent of $R\geq 1$. Here one covers $B_{3R}\setminus B_{2R}$ by balls adapted to the scale $R$ and applies the standard interior elliptic estimate on slightly larger balls. The overlap is uniformly bounded, so the constant is independent of $R\geq 1$. Hence, by the assumption,
\begin{align*}
    \|\m-\e{3}\|_{H^2(B_{3R}\setminus B_{2R})}^2
    \le C\varepsilon .
\end{align*}
By Sobolev embedding, after decreasing $\varepsilon_0$ if necessary, we may assume
\begin{align*}
    |\bar m|\le \frac{1}{\sqrt{2}},
    \qquad
    m_3\geq \frac{1}{\sqrt{2}}
    \qquad
    \text{on } B_{3R}\setminus B_{2R},
\end{align*}
where $\bar m=(m_1,m_2)$. Thus on this annulus
\begin{align*}
    \m
    =
    \bigl(\bar m,\sqrt{1-|\bar m|^2}\,\bigr).
\end{align*}

Let $\eta\in C^\infty([0,\infty))$ satisfy
\begin{align*}
    0\le \eta\le 1,
    \qquad
    \eta(t)=1 \text{ for } t\le 2
    \quad \text{ and } \quad
    \eta(t)=0 \text{ for } t\geq 3 .
\end{align*}
Set
\begin{align*}
    \eta_R(x)=\eta\left(\frac{|x|}{R}\right).
\end{align*}
Then
\begin{align*}
    |\nabla^j\eta_R|\le C R^{-j},
    \qquad
    j=1,2 .
\end{align*}
We define
\begin{align*}
    \m^{(1)}(x)
    =
    \begin{cases}
        \m(x),
        & x\in B_{2R},
        \\
        \left(
        \eta_R(x)\bar m(x),
        \sqrt{1-|\eta_R(x)\bar m(x)|^2}
        \right),
        & x\in B_{3R}\setminus B_{2R},
        \\
        \e{3},
        & x\in \R^3\setminus B_{3R}.
    \end{cases}
\end{align*}
This is well defined because $|\eta_R\bar m|\le 1/\sqrt{2}$. Moreover
\begin{align*}
    \m^{(1)}=\m \text{ on } B_{2R},
    \qquad
    \m^{(1)}=\e{3} \text{ on } \R^3\setminus B_{3R}.
\end{align*}
Let
\begin{align*}
    \Psi(u)=\bigl(u,\sqrt{1-|u|^2}\,\bigr).
\end{align*}
On $\{|u|\le 1/\sqrt{2}\}$ the map $\Psi$ has bounded $C^2$ norm. Since
$\m^{(1)}=\e{3}$ on $\R^3\setminus B_{3R}$, we have
\begin{align*}
    \|\m^{(1)}-\e{3}\|_{H^2(\R^3\setminus B_{2R})}^2
    =
    \|\m^{(1)}-\e{3}\|_{H^2(B_{3R}\setminus B_{2R})}^2 
\end{align*}
and $\m^{(1)} - \e{3} = \Psi(\eta_R \bar m) - \Psi(0)$ on $B_{3R}\setminus B_{2R}$. Therefore, using the bounds on $\eta_R$ and $R\geq 1$, we have
\begin{align*}
    \|\m^{(1)}-\e{3}\|_{H^2(\R^3\setminus B_{2R})}^2
    &\le
    C\|\eta_R\bar m\|_{H^2(B_{3R}\setminus B_{2R})}^2 \le
    C\|\bar m\|_{H^2(B_{3R}\setminus B_{2R})}^2
    \\
    &\le
    C\|\m-\e{3}\|_{H^2(B_{3R}\setminus B_{2R})}^2 \le C\varepsilon .
\end{align*}
This proves the desired estimate for $\m^{(1)}$.
\end{proof}
We now combine the cut-off construction with the Whitehead representation to show that dichotomy induces a nontrivial splitting of both the energy and the Hopf charge.
\begin{lemma}\label{lemma:dichotomy}
Let $H>\frac14$ and $k \in \Z \setminus \{0\}$. Assume further there exists a minimizing sequence $(\m_n)_{n \in \N}$ for $I_k(H)$ such that condition~\ref{case:dichotomy} dichotomy is satisfied. Then there exists $\ell \in \Z \setminus\{0,k\}$ with
\begin{align*}
I_k(H) \geq I_\ell(H) + I_{k-\ell}(H).
\end{align*}
\end{lemma}
\begin{proof}
Suppose that dichotomy~\ref{case:dichotomy} occurs for $(e(\m_n))_{n\in \N}$. Let $0 < \lambda < 1$ be the $\lambda$ described in dichotomy and fix $\varepsilon > 0$ to find $R, R_n, y_n$ as stated in~\ref{case:dichotomy}. Assume further $R>1$ and by translation invariance $y_n = 0$, then 
\begin{align*}
|\lambda\E_H[\m_n] - \E_H[\m_n;B_R]| < \varepsilon \text{ and } |(1-\lambda)\E_H[\m_n] - \E_H[\m_n;\R^3 \setminus B_{R_n}]| < \varepsilon.
\end{align*}
Without loss of generality $R_n > 4R$ since $R_n \to \infty$, then we conclude
\begin{align*}
&|(1-\lambda)\E_H[\m_n] - \E_H[\m_n;\R^3 \setminus B_{4R}] | < 3 \varepsilon \text{ and} \\
&\E_H[\m_n;B_{4R} \setminus B_R] < 2 \varepsilon.
\end{align*}
Both statements can be shown by showing that $\E_H[\m_n;B_{R_n} \setminus B_R] < 2\varepsilon$, a consequence of
\begin{align*}
\E_H[\m_n;B_{R_n} \setminus B_R] 
&\le |\E_H[\m_n;B_{R_n}] - \lambda \E_H[\m_n]| + |\lambda \E_H[\m_n] - \E_H[\m_n;B_R]|\\
&= |(1-\lambda)\E_H[\m_n] - \E_H[\m_n;\R^3 \setminus B_{R_n}] | \\
&+ |\lambda \E_H[\m_n] - \E_H[\m_n;B_R]|
\end{align*}
and~\ref{case:dichotomy} dichotomy. In particular, this shows that in the dichotomy case the radius $R$ can be replaced by a radius $R'\in(R,R_n)$, and the same conclusions remain valid. Since $R_n\to\infty$, we may therefore enlarge $R$ in the proof whenever needed. In particular, we assume $R\ge 1$. Then, by \eqref{eq:densityCoercivity}, we get that
\begin{align*}
\int_{B_{4R}\setminus B_R} |\m - \e 3|^2 + |\Delta \m|^2 \d x< c_0 \varepsilon.
\end{align*}
Let $\m^{(1)}_n,\m^{(2)}_n\in H^2(\R^3;\S^2)$ be the maps provided by Lemma~\ref{cut-offconstruction}, such that
\begin{align*}
&\m^{(1)}_n = \m_n \text{ on } B_R ,&& \m^{(2)}_n = \m_n \text{ on } \R^3 \setminus B_{4R}, \\
&\m^{(1)}_n = \e{3} \text{ on } \R^3 \setminus B_{4R}  , &&\m^{(2)}_n = \e{3} \text{ on } B_R.
\end{align*}
From the same Lemma and $\E_H[\m_n;B_{4R} \setminus B_R] < c\varepsilon$ we get that 
\begin{align*}
&\lVert \m^{(1)}_n - \e{3} \rVert_{H^2(\R^3 \setminus B_R)}^2 + \lVert \m^{(2)}_n - \e{3} \rVert_{H^2(B_{4R})}^2\\ 
&=  \lVert \m^{(1)}_n - \e{3} \rVert_{H^2(B_{4R} \setminus B_R)}^2 + \lVert \m^{(2)}_n - \e{3} \rVert_{H^2(B_{4R} \setminus B_R)}^2 < c \varepsilon . 
\end{align*}
This implies 
\begin{align*}
\E_H[\m^{(1)}_n;B_{4R} \setminus B_R] < c \varepsilon \text{ and } \E_H[\m^{(2)}_n;B_{4R} \setminus B_R] < c \varepsilon.
\end{align*}
Let $\F_n$ and $\F_n^{(j)}$ denote the fields corresponding to $\m_n$ and $\m_n^{(j)}$, respectively. Choose potentials in Coulomb gauges
\begin{align*}
\nabla\times \A_n=\F_n,
\qquad
\nabla\times \A_n^{(j)}=\F_n^{(j)}
\end{align*}
for $j=1,2$. Then
\begin{align*}
\nabla\times\bigl(\A_n-\A_n^{(1)}-\A_n^{(2)}\bigr)
=
\F_n-\F_n^{(1)}-\F_n^{(2)},
\qquad
\nabla\cdot\bigl(\A_n-\A_n^{(1)}-\A_n^{(2)}\bigr)=0.
\end{align*}
We compute
\begin{align*}
&16\pi^2
\left|Q[\m_n]-Q[\m_n^{(1)}]-Q[\m_n^{(2)}]\right|
\\
&\le
\left|
\int_{\R^3}
\bigl(\F_n-\F_n^{(1)}-\F_n^{(2)}\bigr)\cdot\A_n
\d x
\right|
+
\left|
\int_{\R^3}
\F_n^{(1)}\cdot
\bigl(\A_n-\A_n^{(1)}-\A_n^{(2)}\bigr)
\d x
\right|
\\
&+
\left|
\int_{\R^3}
\F_n^{(2)}\cdot
\bigl(\A_n-\A_n^{(1)}-\A_n^{(2)}\bigr)
\d x
\right|
+
\left|
\int_{\R^3}\F_n^{(1)}\cdot\A_n^{(2)}\d x
\right|
+
\left|
\int_{\R^3}\F_n^{(2)}\cdot\A_n^{(1)}\d x
\right|.
\end{align*}
Since $\F_n-\F_n^{(1)}-\F_n^{(2)}$ is supported in $B_{4R}\setminus B_R$, the first three terms are estimated by Hölder and Sobolev as
\begin{align*}
\left|
\int_{\R^3}
\bigl(\F_n-\F_n^{(1)}-\F_n^{(2)}\bigr)\cdot\A_n
\d x
\right|
&\le cM^{\frac12}\varepsilon^{\frac56},
\\
\left|
\int_{\R^3}
\F_n^{(j)}\cdot
\bigl(\A_n-\A_n^{(1)}-\A_n^{(2)}\bigr)
\d x
\right|
&\le cM^{\frac56}\varepsilon^{\frac12}
\end{align*}
for $j=1,2$. It remains to estimate the mixed terms. We only consider the first one. Since $\F_n^{(1)}=0$ outside $B_{4R}$,
\begin{align*}
\left|
\int_{\R^3}\F_n^{(1)}\cdot\A_n^{(2)}\d x
\right|
&\le
\left|
\int_{B_R}\F_n^{(1)}\cdot\A_n^{(2)}\d x
\right|
+
\left|
\int_{B_{4R}\setminus B_R}\F_n^{(1)}\cdot\A_n^{(2)}\d x
\right|
\\
&\le
\left|
\int_{B_R}\F_n^{(1)}\cdot\A_n^{(2)}\d x
\right|
+
cM^{\frac12}\varepsilon^{\frac56}.
\end{align*}
For $x\in B_R$ we use the Biot--Savart representation only for the contribution of
$\F_n^{(2)}$ outside $B_R$. The part coming from $B_{4R}\setminus B_R$ gives
$cM^{\frac56}\varepsilon^{\frac12}$. For the part coming from
$\R^3\setminus B_{4R}$ we use $|x-y|\ge 3R$ and obtain
\begin{align*}
\left|
\int_{B_R}\F_n^{(1)}\cdot\A_n^{(2)}\d x
\right|
&\le
cM^{\frac56}\varepsilon^{\frac12}
+
\frac{c}{R^2}
\|\F_n^{(1)}\|_{L^1(\R^3)}
\|\F_n^{(2)}\|_{L^1(\R^3)}
\\
&\le
cM^{\frac56}\varepsilon^{\frac12}
+
\frac{cM^2}{R^2}.
\end{align*}
The same estimate holds for
$\int_{\R^3}\F_n^{(2)}\cdot\A_n^{(1)}\d x$. Hence
\begin{align*}
\left|Q[\m_n]-Q[\m_n^{(1)}]-Q[\m_n^{(2)}]\right|
\le
cM^{\frac12}\varepsilon^{\frac56}
+
cM^{\frac56}\varepsilon^{\frac12}
+
\frac{cM^2}{R^2}.
\end{align*}
Choosing $\varepsilon>0$ small and $R$ large, the right-hand side is smaller than $1$. Since all three Hopf invariants are integers, we obtain
\begin{align}\label{eq:chargesplit}
Q[\m_n]=Q[\m_n^{(1)}]+Q[\m_n^{(2)}].
\end{align}
Note that we can choose $\varepsilon$ as stated, because the constant $c$ is independent of $n, R$ and $R_n$. Observe further that 
\begin{align*}
|Q[\m^{(1)}_n]| \le c \E_H[\m^{(1)}_n]^{\frac43} \le c(\E_H[\m_n] + \E_H[\m^{(1)}_n;B_{4R} \setminus B_R])^{\frac43} \le c(M+\varepsilon)^{\frac43}.
\end{align*}
After passing to a further subsequence, we conclude $Q[\m^{(1)}_n]$ is constant. Denoting this constant by $\ell$, \eqref{eq:chargesplit} implies $Q[\m_n^{(2)}]=k-\ell$. If $\ell=k$, then $I_k(H) \le \E_H[\m^{(1)}_n]$ for every $n \in \N$. In particular
\begin{align*}
I_k(H) &\le \E_H[\m^{(1)}_n] = \E_H[\m_n;B_R] + \E_H[\m^{(1)}_n; B_{4R} \setminus B_R]\\
&\le |\lambda \E_H[\m_n] - \E_H[\m_n;B_R]| + \lambda \E_H[\m_n] + \E_H[\m^{(1)}_n; B_{4R} \setminus B_R]\\
&\le \lambda \E_H[\m_n]  + c \varepsilon \to \lambda I_k(H) + c \varepsilon.
\end{align*}
Choosing $\varepsilon < \frac{1-\lambda}{c}I_k(H)$ contradicts this case. If $\ell=0$ then $Q[\m^{(2)}_n]=k$, hence we would derive a similar contradiction. Therefore $\ell \in \Z \setminus \{ 0 , k \}$ and we compute (for $n \in \N$ large enough)
\begin{align*}
\E_H[\m_n] &= \lambda \E_H[\m_n]  + (1 - \lambda ) \E_H[\m_n] \\
&\geq \E_H[\m_n;B_R] + \E_H[\m_n ;\R^3 \setminus B_{4R}] - c \varepsilon \\
&= \E_H[\m^{(1)}_n] + \E_H[\m^{(2)}_n] -\E_H[\m^{(1)}_n;B_{4R}\setminus B_{R}] - \E_H[\m^{(2)}_n;B_{4R}\setminus B_{R}]-   c \varepsilon  \\
&\geq \E_H[\m^{(1)}_n] + \E_H[\m^{(2)}_n] - c \varepsilon\\
&\geq I_\ell(H) + I_{k-\ell}(H) - c \varepsilon\\
\end{align*}
and with $n \to \infty$
\begin{align*}
I_k(H) \geq I_\ell(H) + I_{k-\ell}(H) - c \varepsilon.
\end{align*}
Although $\ell$ may depend on $\eps$, it is integer-valued and uniformly bounded. Thus, along a sequence $\eps_j\to 0$, we may assume that $\ell$ is constant. Letting $j\to\infty$ proves the claim.
\end{proof}	

\begin{prop}\label{prop:existence}
Let $H>\frac14$ and $k \in \Z$. If for all $\ell \in \Z \setminus \{ 0,k \}$ 
\begin{align*}
I_k(H) < I_\ell(H) + I_{k-\ell}(H),
\end{align*}
then $I_k(H)$ is attained.
\end{prop}
\begin{proof}
Let $(\m_n)_{n \in \N} \subset H^2(\R^3;\S^2)$ be a minimizing sequence and assume without loss of generality $k \neq 0$. Using Lemma~\ref{lemma:vanishing} and Lemma~\ref{lemma:dichotomy}, $(e(\m_n))_{n \in \N}$, up to some subsequence, can only satisfy the compactness condition~\ref{case:compactness}. By translation invariance of $\mathcal E$ and $Q$ we may assume the concentrating centers are at the origin, i.e. for every $\varepsilon>0$ there exists $R>0$ such that
\begin{align} \label{eq:tightness}
\int_{\mathbb R^3\setminus B_R} e(\m_n)\d x \le \varepsilon
\quad \text{for all large } n .
\end{align}
The coercivity estimate \eqref{eq:densityCoercivity}, together with the elliptic estimate on $\R^3$, gives
\begin{align*}
\|\m_n - \e{3}\|_{H^2(\mathbb R^3)}^2 \le C \E_H[\m_n] \le C M .
\end{align*}
Since $H^2$ is reflexive there is a map $\m:\R^3 \to \R^3$ such that $\m-\e 3\in H^2(\mathbb R^3;\mathbb R^3)$ with
\begin{align*}
\m_n - \e{3} \rightharpoonup \m - \e{3} \quad \text{in } H^2(\mathbb R^3) .
\end{align*}
By Rellich on bounded sets together with \eqref{eq:tightness}, we also have
\begin{align*}
\m_n - \e{3} \to \m - \e{3} \quad \text{in } H^1(\mathbb R^3).
\end{align*}
and hence, up to a subsequence,
\begin{align*}
\m_n \to \m \quad \text{and} \quad \nabla \m_n \to \nabla \m \quad \text{almost everywhere in } \mathbb R^3 .
\end{align*}
In particular $\m(x)\in \mathbb S^2$ almost everywhere, hence $\m\in H^2(\mathbb R^3;\mathbb S^2)$. Interpolating the strong $H^1$ convergence with the weak $H^2$ convergence also yields
\begin{align*}
    \|\m_n - \m\|_{H^s(\R^3)} \le\|\m_n - \m\|_{H^1(\R^3)}^{2-s}  \|\m_n - \m\|_{H^2(\R^3)}^{s-1}\to 0
\end{align*}
as $n\to \infty$ for every $1<s<2$. Furthermore, for $s>\frac32$, we conclude uniform convergence, hence $Q[\m]=Q[\m_n]=k$ for large $n \in \N$. Lastly, the energy is weakly lower semicontinuous
\begin{align*}
\E_H[\m] \le \liminf_{n\to\infty} \E_H[\m_n] = I_k(H) .
\end{align*}
\end{proof}
\begin{cor}\label{cor:minimalenergy}
Let $H>\frac14$. There exists a minimal energy level, that is $k_0=k_0(H)\in\Z\setminus\{0\}$ with
\begin{align*}
I_{k_{0}}(H)=\inf\{I_k(H):k\in\Z\setminus\{0\}\},
\end{align*}
and this infimum is attained.
\end{cor}
\begin{proof}
Choose $ k_1\in\mathbb{N}$ with $I_1(H)<C(H) k_1^{\frac34}$. Then for all $|\ell|> k_1$,
\begin{align*}
I_\ell(H)\ge C (H)|\ell|^{\frac34}>I_1(H).
\end{align*}
Hence
\begin{align*}
\inf\limits_{k \in \Z \setminus \{ 0 \} } I_k(H) = \inf\limits_{k \in \{ -k_1, \dots , k_1 \} \setminus \{ 0 \} }  I_k(H).
\end{align*}
so there is $k_{0}\in\mathbb{Z}\setminus\{0\}$ with $I_{k_{0}}(H)=\inf_{k\ne0}I_k(H)$. Therefore, for $\ell \in \Z \setminus \{0,k_0\}$, we get
\begin{align*}
I_{k_0}(H) < 2 I_{k_0}(H) \le I_\ell(H) +I_{k_0-\ell}(H).
\end{align*}
With Proposition~\ref{prop:existence} we obtain the claim.
\end{proof}

\subsection{Attainment in infinitely many homotopy classes}

Define
\begin{align} \label{eq:SigmaH}
\Sigma_H = \{ k \in \Z\setminus \{0\} : I_k(H) <I_{\ell}(H) + I_{k-\ell}(H) \text{ for every } \ell \in \Z\setminus \{0,k\} \}.
\end{align}

\begin{lemma}\label{EnDec}
Let $H>\frac14$. For every $k \in \Z\setminus\{0\}$ there exist $J \in \N$ and $d_{1},\dots,d_{J} \in \Sigma_H$ with $k=\sum_{i=1}^{J} d_{i}$ such that
\begin{align*}
I_k(H) \ge \sum_{i=1}^{J} I_{d_i}(H).
\end{align*}
\end{lemma}
\begin{proof}
Fix $k \in \Z$. If $k \in \Sigma_H$ the statement is immediate with $J=1$ and $d_{1}^{(1)}=k$. Assume $k \notin \Sigma_H$. By \eqref{eq:SigmaH} there exist nonzero integers $d^{(2)}_{1},d^{(2)}_{2}$ with
\begin{align*}
k = d^{(2)}_{1} + d^{(2)}_{2}, \qquad I_k(H) \ge I_{d^{(2)}_{1}}(H) + I_{d^{(2)}_{2}}(H).
\end{align*}
If both $d^{(2)}_{1}$ and $d^{(2)}_{2}$ lie in $\Sigma_H$, the proof is complete. Otherwise, choose the smallest $i \in \{1,2\}$ with $d^{(2)}_{i}\notin \Sigma_H$. We then apply \eqref{eq:SigmaH} again to decompose $d^{(2)}_{i}$ into two nonzero integers. After $J$ refinement steps this yields a finite family of nonzero integers
\begin{align*}
\left( d^{(J)}_{i} \right)_{1 \le i \le J}
\end{align*}
such that
\begin{align*}
k = \sum_{i=1}^{J} d^{(J)}_{i}, \qquad
I_k(H) \ge  \sum_{i=1}^{J} I_{d^{(J)}_i}(H).
\end{align*}
Using
\begin{align*}
I_d(H) \ge C(H)|d|^{\frac34} \ge C(H) \quad \text{for every nonzero } d \in \Z,
\end{align*}
we infer
\begin{align*}
I_k(H) \ge JC(H).
\end{align*}
By Lemma~\ref{lemma:energygrowth}, $I_k(H) \le C|k|^{\frac34}$, which implies
\begin{align*}
J \le \frac{I_k(H)}{C(H)}  \le \frac{C}{C(H)}|k|^{\frac34}.
\end{align*}
Thus the refinement terminates after finitely many steps, which proves the claim.
\end{proof}
\begin{prop}\label{prop:SigmaHunbound}
For every $H>\frac14$, there exist infinitely many $k \in\Sigma_H$.
\end{prop}
\begin{proof}
Fix $H>\frac14$ and assume for contradiction that $|\Sigma_H|<\infty$. Set
\begin{align*}
K = \max_{d \in \Sigma_H} |d|.
\end{align*}
Let $k \in \Z$. By Lemma~\ref{EnDec} there exist $J \in \N$ and $d_1,\dots,d_J \in \Sigma_H$ with
\begin{align*}
k=\sum_{i=1}^{J} d_{i}, \qquad I_k(H) \ge \sum_{i=1}^{J} I_{d_i}(H).
\end{align*}
The topological lower bound implies $I_{d_i}(H) \ge C(H)$ for every $1 \le i\le J$, hence
\begin{align*}
I_k(H) \ge JC(H).
\end{align*}
On the other hand,
\begin{align*}
|k| \le \sum_{i=1}^{J} |d_{i}| \le K J \le \frac{K}{C(H)} I_k(H).
\end{align*}
Using Lemma~\ref{lemma:energygrowth}, $I_k(H) \le C|k|^{\frac34}$, we obtain
\begin{align*}
|k| \le \frac{K C}{C(H)}|k|^{\frac34}.
\end{align*}
This gives a contradiction for large $|k|$. Therefore $|\Sigma_H|=\infty$.
\end{proof}
\section{The large-field regime}

\begin{lemma}\label{lemma:scalinglimit}
Let $\m\in H^2(\R^3;\S^2)$. Then
\begin{align}
    \label{eq:scalingLimitFormula}
    \lim_{H \to \infty} \frac{1}{C(H)} \inf_{\lambda >0} \E_H[ \m^{\lambda}] = \frac{1}{2^4 3^{\frac38}\pi^2}\| \m -\e{3} \|_{L^2}^{\frac12}\| \Delta \m \|_{L^2}^{\frac32}.
\end{align}
\end{lemma}
\begin{proof}
The claim follows by using Lemma~\ref{lemma:scaling}
\begin{align*}
\inf_{\lambda>0}\E_H[\m^{\lambda}] = \frac{\sqrt{2}}{3} (12 H)^{\frac14} \|\m-\e{3}\|_{L^2}^{\frac12}
\|\Delta\m\|_{L^2}^{\frac32}(1+o_H(1))
\end{align*}
and
\begin{align*}
C(H) =  \frac{2^{\frac{9}{2}} \pi^2}{3^{\frac58}}(12H)^{\frac14}(1+o_H(1))
\end{align*}
as $H \to \infty$. 
\end{proof}

\subsection{An explicit unit-charge competitor}

To construct explicit representatives in a fixed homotopy class, we use the standard factorization through the Hopf fibration. Namely, we choose a map
\begin{align*}
    \Phi:\R^3\cup\{\infty\}\simeq \S^3 \to \S^3
\end{align*}
and define
\begin{align*}
    \m
    =
    \pi_{\mathrm{Hopf}}\circ\Phi .
\end{align*}
With our normalization of the Hopf fibration
\begin{align*}
    Q[\m]
    =
    \deg\Phi .
\end{align*}
In particular, a degree-one choice of $\Phi$ gives a Hopf configuration with
\begin{align*}
    Q[\m]=1.
\end{align*}
We use the Hopf projection
\begin{align*}
    \pi_{\mathrm{Hopf}}:\S^3\subset\C^2\to\S^2
\end{align*}
given by
\begin{align*}
    \pi_{\mathrm{Hopf}}(Z_1,Z_2)
    =
    \left(
    2 \Re(Z_1\overline{Z_2}),
    2\Im(Z_1\overline{Z_2}),
    |Z_2|^2-|Z_1|^2
    \right).
\end{align*}
We combine this projection with a spinor ansatz (cf. e.g. \cite{Sutcliff2007Knots})
\begin{align*}
    Z_1(x)
    &=
    \frac{x_1+ix_2}{r}\sin f(r),
    \\
    Z_2(x)
    &=
    \frac{x_3}{r}\sin f(r)
    +
    i\cos f(r),
    \qquad
    r=|x|.
\end{align*}
Assuming $f(0),f(\infty)\in\pi\Z$, we conclude $Z=(Z_1,Z_2)$ extends continuously to
\begin{align*}
    Z:\R^3\cup\{\infty\}\simeq\S^3\to\S^3.
\end{align*}
The corresponding map into the sphere is
\begin{align*}
\m(x) =\pi_{\mathrm{Hopf}}(Z_1(x),Z_2(x)).
\end{align*}
In the canonical identification $\C^2\simeq\R^4$, we have
\begin{align*}
\det \left( Z, D_{(r,\theta,\varphi)} Z \right)=-f'(r)\sin^2 f(r)\sin\theta.
\end{align*}
Thus
\begin{align*}
\deg Z = \frac{2}{\pi} \int_0^\infty -f'(r)\sin^2 f(r)\,\mathrm{d}r = \frac{f(0)-f(\infty)}{\pi}.
\end{align*}
In particular,
\begin{align*}
f(0)=\pi, \qquad f(\infty)=0
\end{align*}
gives
\begin{align*}
    Q[\m]=\deg Z=1.
\end{align*}
Writing
\begin{align*}
    f(r)=2\arctan h(r)
\end{align*}
gives
\begin{align*}
    \sin f(r) = \frac{2h(r)}{1+h(r)^2}, \qquad \cos f(r) = \frac{1-h(r)^2}{1+h(r)^2}.
\end{align*}
The unit Hopf boundary conditions become
\begin{align*}
    h(0)=\infty,
    \qquad
    h(\infty)=0.
\end{align*}
Indeed, these conditions imply
\begin{align*}
    f(0)=\pi,
    \qquad
    f(\infty)=0,
\end{align*}
and hence $\deg Z=1.$

The standard Hopf map is obtained from
\begin{align*}
    h(r)=\frac{1}{r},
    \qquad
    f(r)=2\arctan\frac{1}{r}.
\end{align*}

However, this exact Hopf map is not suitable for the energy because of slow deacy $\m-\e{3}\notin L^2(\R^3).$ 

We now choose the explicit finite-energy modification of the standard Hopf profile
\begin{align*}
    h(r)
    =
    \frac{e^{-r}}{r}.
\end{align*}
The resulting map has Hopf index one and belongs to the admissible energy class, as follows from the behavior of the profile at the origin and its exponential decay at infinity.
The map $\m_\ast$ associated with the profile $h(r)=e^{-r}/r$ satisfies (see appendix~\ref{sec:appendix})
\begin{align}\label{eq:competitor}
    \frac{1}{2^4 3^{\frac38}\pi^2}
    \|\m_\ast-\e{3}\|_{L^2(\R^3)}^{\frac12}
    \|\Delta\m_\ast\|_{L^2(\R^3)}^{\frac32}
    <
    2^{\frac34}+3^{\frac34}.
\end{align}

\begin{proof}[Proof of Theorem~\ref{thm:largeHunit}]
Part~\ref{item:infinitelymany} follows from Proposition~\ref{prop:SigmaHunbound} and by \eqref{eq:SigmaH} together with Proposition~\ref{prop:existence}. For Part~\ref{item:largeHexistence} it suffices to treat the class $k=1$. By Proposition~\ref{prop:existence}, $I_1(H)$ is attained provided
\begin{align*}
    I_1(H) < I_\ell(H)+I_{1-\ell}(H)
    \qquad\text{for all } \ell\in\Z\setminus\{0,1\}.
\end{align*}
For $\ell\in \{-1,2\}$ this is immediate
\begin{align*}
    I_1(H)<I_1(H)+I_2(H)=I_{-1}(H)+I_2(H).
\end{align*}
It therefore remains to verify the inequality for $\ell\in\Z\setminus\{-1,0,1,2\}$. For such $\ell$ the topological lower bound $I_k(H)\ge C(H)|k|^{\frac34}$ \cite{harland} gives
\begin{align*}
    I_\ell(H)+I_{1-\ell}(H)
    \ge C(H)\bigl(|\ell|^{\frac34}+|1-\ell|^{\frac34}\bigr)
    \ge C(H)\bigl(2^{\frac34}+3^{\frac34}\bigr).
\end{align*}
Hence it suffices to show
\begin{align*}
    \frac{I_1(H)}{C(H)} < 2^{\frac34}+3^{\frac34}
    \qquad\text{for all sufficiently large } H.
\end{align*}
Let $\m_\ast$ be the unit charge competitor of the preceding subsection. Since $\m_\ast^\lambda$ has Hopf index one for every $\lambda>0$, we have $I_1(H)\le\inf_{\lambda>0}\E_H[\m_\ast^\lambda]$, and Lemma~\ref{lemma:scalinglimit} yields
\begin{align*}
    \limsup_{H\to\infty}\frac{I_1(H)}{C(H)}
    \le \lim_{H\to\infty}\frac{1}{C(H)}\inf_{\lambda>0}\E_H[\m_\ast^\lambda]
    = \frac{1}{2^4 3^{\frac38}\pi^2}\|\m_\ast-\e{3}\|_{L^2}^{\frac12}\|\Delta\m_\ast\|_{L^2}^{\frac32}.
\end{align*}
By \eqref{eq:competitor} the right-hand side is strictly below $2^{\frac34}+3^{\frac34}$. Thus there exists $H_0>\tfrac14$ such that $I_1(H)/C(H)<2^{\frac34}+3^{\frac34}$ for all $H\ge H_0$, the strict subadditivity condition holds, and $I_1(H)$ is attained.
\end{proof}

\begin{rem}[Quantitative threshold]
The rigorous numerical estimates in Appendix~\ref{app:threshold} show that one may choose $H_0=7.69$ in Theorem~\ref{thm:largeHunit}.
\end{rem}

\subsection{Large-field asymptotics of minimizers}
For every $H\geq H_0$, let $\m_H$ be a minimizer of $\E_H$ in the class of maps with Hopf index one. The upper bound from Lemma~\ref{lemma:energygrowth} and the lower bound obtained by Harland in~\cite{harland} yield that $\E_H[\m_H]$ is of order $H^{\frac14} $ as $H\to\infty.$

\begin{lemma}\label{lemma:asymptotic}
There exists $c,C>0$, $H_0>\frac14$, such that
\begin{align*} 
c H^{\frac14} \le \|\Delta\m_H\|_{L^2(\R^3)}^2
    \le C H^{\frac14} \quad \text{ and } \quad
    c H^{-\frac34} \le \|\m_H-\e{3}\|_{L^2(\R^3)}^2
    \le C H^{-\frac34}.
    \end{align*}
for every $H \ge H_0$. In particular, for $H \geq H_0$,
\begin{align*}
    c H^{-\frac14} \le \|\nabla\m_H\|_{L^2(\R^3)}^2
    \le C H^{-\frac14}.
\end{align*}
\end{lemma}

\begin{proof}
By coercivity, both terms are controlled by the energy:
\begin{align*}
    \|\Delta\m_H\|_{L^2(\R^3)}^2
    \leq
    C\E_H[\m_H]
\end{align*}
and
\begin{align*}
    H\|\m_H-\e{3}\|_{L^2(\R^3)}^2
    \leq
    C\E_H[\m_H].
\end{align*}
Since
\begin{align*}
    \E_H[\m_H]\leq CH^{\frac14},
\end{align*}
we immediately obtain
\begin{align*}
    \|\Delta\m_H\|_{L^2(\R^3)}^2
    \leq
    CH^{\frac14},
    \qquad
    \|\m_H-\e{3}\|_{L^2(\R^3)}^2
    \leq
    CH^{-\frac34}.
\end{align*}

It remains to show the corresponding lower bounds. Since $Q[\m_H]=1$, Lemma~\ref{lemma:scaling_invariant_topological_bound} gives a constant $c_0>0$, independent of $H$, such that
\begin{align*}
    \|\m_H-\e{3}\|_{L^2(\R^3)}^{\frac12}
    \|\Delta\m_H\|_{L^2(\R^3)}^{\frac32}
    \geq
    c_0.
\end{align*}
Combining the lower bound with the corresponding upper bounds yields the lower bounds in terms of $H$. The usual interpolation (integration by parts and Hölder's inequality) together with \eqref{eq:intermediateVK} yield the corresponding bounds of $\| \nabla \m_H\|_{L^2}$.
\end{proof}

Consequently,
\begin{align*}
    \m_H-\e{3}\to 0
    \qquad\text{strongly in }H^1(\R^3;\R^3).
\end{align*}
On the other hand,
\begin{align*}
    \|\Delta\m_H\|_{L^2(\R^3)}^2
    \geq
    cH^{\frac14}\to\infty,
\end{align*}
so the sequence does not converge in $H^2$. Thus, up to translations, the nontrivial energy concentrates on a shrinking length scale.

Define
\begin{align*}
    \widetilde{\m}_H(x)
    =
    \m_H\bigl(H^{-\frac14}x\bigr).
\end{align*}
This is the natural scaling on which both
\begin{align*}
    \|\Delta\widetilde{\m}_H\|_{L^2(\R^3)}^2
    \qquad\text{and}\qquad
    \|\widetilde{\m}_H-\e{3}\|_{L^2(\R^3)}^2
\end{align*}
remain of order one. In particular, this is the scale on which one expects a nontrivial limiting profile.

The corresponding rescaled energy is
\begin{align*}
    \tE_H(\m)
    =
    \frac12
    \int_{\R^3}
    \left(
        |\Delta\m|^2
        -
        {H^{-\frac12}}|\nabla\m|^2
        +
        |\m-\e{3}|^2
    \right)
    \,\d{x}.
\end{align*}

Formally letting $H\to\infty$ leads to the limiting energy
\begin{align*}
    \E^\infty(\m)
    =
    \frac12
    \int_{\R^3}
    \left(
        |\Delta\m|^2
        +
        |\m-\e{3}|^2
    \right)
    \,\d{x}.
\end{align*}

The associated variational problem for $k\in\Z$ is
\begin{align*}
    I^\infty_k
    =
    \inf
    \left\{
        \E^\infty(\m)
        \,\middle|\,
        \m\in H^2(\R^3;\S^2),
        \ Q[\m]=k
    \right\}.
\end{align*}
Since $Q[\m^\lambda]=Q[\m]$ for every $\lambda>0$, we may first minimize with respect to the scaling parameter. This gives
\begin{align}
    \label{eq:Iinftyscaleinvariant}
    I^\infty_k
    =
    \frac{2}{3^{\frac34}}
    \inf
    \left\{
        \|\m-\e{3}\|_{L^2}^{\frac12}
        \|\Delta\m\|_{L^2}^{\frac32}
        \,\middle|\,
        \m\in H^2(\R^3;\S^2),
        \ Q[\m]=k
    \right\}.
\end{align}

\begin{lemma}\label{lemma:energyasymptotic}
For every $H>\frac14$, we have
\begin{align*}
    \frac{\sqrt{1+(12H)^{-1}}-2(12H)^{-\frac12}}
    {\sqrt{\sqrt{1+(12H)^{-1}}-(12H)^{-\frac12}}}
    I_k^\infty
    \le
    \frac{I_k(H)}{H^{\frac14}}
    \le
    I_k^\infty.
\end{align*}
In particular,
\begin{align*}
    \lim_{H\to\infty}\frac{I_k(H)}{H^{\frac14}}
    =
    I_k^\infty.
\end{align*}
\end{lemma}

\begin{proof}
The estimates follow directly from Lemma~\ref{lemma:scaling} and
\eqref{eq:Iinftyscaleinvariant}.
\end{proof}

\begin{lemma}\label{lemma:limitproblem}
For every sequence $H_j\to\infty$, the rescaled sequence $(\widetilde{\m}_{H_j})_{j\in\N}$ is a minimizing sequence for $I^\infty_1$.
\end{lemma}
\begin{proof}
    Let $H_j\to\infty$ be arbitrary and set
\begin{align*}
    \bs{u}_j:=\widetilde{\m}_{H_j}.
\end{align*}
By the definitions of $\tE_{H_j}$ and $\E^\infty$, we have
\begin{align}
    \E^\infty(\bs{u}_j)
    &=
    \tE_{H_j}(\bs{u}_j)
    +
    \frac{1}{2H_j^{\frac12}}
    \|\nabla\bs{u}_j\|_{L^2(\R^3)}^2
    \notag\\
    &=
    \frac{I_1(H_j)}{H_j^{\frac14}}
    +
    \frac{1}{2H_j^{\frac12}}
    \|\nabla\bs{u}_j\|_{L^2(\R^3)}^2.
    \label{eq:rescaled_energy_identity}
\end{align}
Thus, by Lemma~\ref{lemma:asymptotic} and Lemma~\ref{lemma:energyasymptotic}, $(\bs{u}_j)_{j\in\N}$ is a minimizing sequence for $I^\infty_1$.
\end{proof}

The limiting problem shares the structural features of $\E_H$ that were used in Section~\ref{sec:dichotomy}: the density
\begin{align*}
    e^\infty(\m)
    =
    \frac12
    \left(
        |\Delta\m|^2
        +
        |\m-\e{3}|^2
    \right)
\end{align*}
is nonnegative and coercive in the sense of \eqref{eq:densityCoercivity}, and both $\E^\infty$ and $Q$ are invariant under translations. The concentration-compactness scheme of Section~\ref{sec:dichotomy} therefore applies to minimizing sequences of
\begin{align*}
    I^\infty_k
    =
    \inf
    \left\{
        \E^\infty(\m)
        \,\middle|\,
        \m\in H^2(\R^3;\S^2),
        \ Q[\m]=k
    \right\},
    \qquad k\in\Z,
\end{align*}
with $e$ replaced by $e^\infty$. A key feature of the limiting problem is that its exact scaling structure makes the required strict subadditivity for the unit charge a direct consequence of the estimates established before.

Lemma~\ref{lemma:scaling_invariant_topological_bound} with \eqref{eq:Iinftyscaleinvariant} yields the topological lower bound
\begin{align}
    \label{eq:limitTopologicalBound}
    I^\infty_k
    \ge
    c_\infty|k|^{\frac34},
    \qquad
    c_\infty
    :=
    \frac{2^5\pi^2}{3^{\frac38}}
    =
    \lim_{H\to\infty}\frac{C(H)}{H^{\frac14}},
\end{align}
the last identity being immediate from the explicit form of $C(H)$.

\begin{prop}\label{prop:limitcompactness}
Let $(\bs{u}_j)_{j\in\N}\subset H^2(\R^3;\S^2)$ with $Q[\bs{u}_j]=1$ be a minimizing sequence for $I^\infty_1$. Then, after passing to a subsequence, the energy densities $(e^\infty(\bs{u}_j))_{j\in\N}$ satisfy the compactness alternative~\ref{case:compactness}: there exists a sequence of centers $(y_j)_{j\in\N}\subset\R^3$ such that for every $\varepsilon>0$ there is a radius $R>0$ with
\begin{align*}
    \int_{\R^3\setminus B_R(y_j)}
    e^\infty(\bs{u}_j)
    \d{x}
    \le
    \varepsilon
\end{align*}
for $j\in\N$ large enough. In particular, neither vanishing~\ref{case:vanishing} nor dichotomy~\ref{case:dichotomy} occurs.
\end{prop}

\begin{proof}
Since $Q[\bs{u}_j]=1\neq 0$, vanishing~\ref{case:vanishing} is excluded exactly as in Lemma~\ref{lemma:vanishing}.

Assume dichotomy~\ref{case:dichotomy}. By Lemma~\ref{cut-offconstruction} and the argument of Lemma~\ref{lemma:dichotomy}, applied with $e$ replaced by $e^\infty$, there is an $\ell\in\Z\setminus\{0,1\}$ with
\begin{align*}
    I^\infty_1
    \ge
    I^\infty_\ell
    +
    I^\infty_{1-\ell}.
\end{align*}
We rule this out for every such $\ell$. For $\ell\in\{-1,2\}$ we have $I^\infty_{-1}=I^\infty_1$, hence
\begin{align*}
    I^\infty_\ell+I^\infty_{1-\ell}
    =
    I^\infty_1+I^\infty_2
    >
    I^\infty_1.
\end{align*}
For $\ell\in\Z\setminus\{-1,0,1,2\}$, we use \eqref{eq:limitTopologicalBound} to conclude
\begin{align*}
    I^\infty_\ell+I^\infty_{1-\ell}
    \ge
    c_\infty
    \bigl(
        |\ell|^{\frac34}+|1-\ell|^{\frac34}
    \bigr)
    \ge
    c_\infty
    \bigl(
        2^{\frac34}+3^{\frac34}
    \bigr).
\end{align*}
On the other hand, Lemma~\ref{lemma:energyasymptotic} and \eqref{eq:limitTopologicalBound} give
\begin{align*}
    \frac{I^\infty_1}{c_\infty}
    =
    \lim_{H\to\infty}
    \frac{I_1(H)}{H^{\frac14}}
    \frac{H^{\frac14}}{C(H)}
    =
    \lim_{H\to\infty}
    \frac{I_1(H)}{C(H)},
\end{align*}
which is strictly smaller than $2^{\frac34}+3^{\frac34}$ by \eqref{eq:competitor} and the proof of Theorem~\ref{thm:largeHunit}. Hence $I^\infty_1<I^\infty_\ell+I^\infty_{1-\ell}$ in all cases, a contradiction. Only~\ref{case:compactness} remains.
\end{proof}

We can now complete the proof of Theorem~\ref{thm:profile}.

\begin{proof}[Proof of Theorem~\ref{thm:profile}]
Part~\ref{item:profileTrivial} was obtained in the discussion preceding Lemma~\ref{lemma:limitproblem}.

For~\ref{item:profileLimit}, let $H_j\to\infty$. By Lemma~\ref{lemma:limitproblem} the sequence $(\widetilde{\m}_{H_j})_{j\in\N}$ is a minimizing sequence for $I^\infty_1$ with $Q[\widetilde{\m}_{H_j}]=1$, so Proposition~\ref{prop:limitcompactness} provides, after passing to a subsequence, centers $y_j\in\R^3$ such that $\bs{u}_j=\widetilde{\m}_{H_j}(\,\cdot\,-y_j)$ satisfies
\begin{align}
    \label{eq:profileTightness}
    \int_{\R^3\setminus B_R}
    e^\infty(\bs{u}_j)
    \d{x}
    \le\varepsilon
    \qquad\text{for all large } j,
\end{align}
where $R=R(\varepsilon)$, and $\E^\infty(\bs{u}_j)\to I^\infty_1$ by translation invariance.

Since $\|\Delta\bs{u}_j\|_{L^2}^2+\|\bs{u}_j-\e{3}\|_{L^2}^2=2\E^\infty(\bs{u}_j)$ and
\begin{align}
    \label{eq:gradientInterpolation}
    \|\nabla \bs{v}\|_{L^2(\R^3)}^2
    \le
    \|\bs{v}\|_{L^2(\R^3)}
    \|\Delta \bs{v}\|_{L^2(\R^3)}
    \qquad\text{for } \bs{v}\in H^2(\R^3;\R^3),
\end{align}
the elliptic estimate on $\R^3$ shows that $(\bs{u}_j-\e{3})_{j\in\N}$ is bounded in $H^2$. Hence, along a further subsequence, $\bs{u}_j-\e{3}\rightharpoonup\bs{u}-\e{3}$ in $H^2(\R^3;\R^3)$ for some $\bs{u}$, and by Rellich $\bs{u}_j\to\bs{u}$ in $H^1_{\loc}(\R^3)$ and almost everywhere, so that $\bs{u}\in H^2(\R^3;\S^2)$. Combining \eqref{eq:profileTightness} with the local convergence gives $\bs{u}_j-\e{3}\to\bs{u}-\e{3}$ in $L^2(\R^3)$, and \eqref{eq:gradientInterpolation} applied to $\bs{v}=\bs{u}_j-\bs{u}$ upgrades this to strong convergence in $H^1(\R^3)$.

Exactly as in the proof of Proposition~\ref{prop:existence}, the weak convergence in $H^2(\R^3)$ together with the strong convergence in $H^1(\R^3)$ yields uniform convergence, hence $Q[\bs{u}]=1$. Weak lower semicontinuity then yields
\begin{align*}
    \E^\infty(\bs{u}) \le \liminf_{j\to\infty} \E^\infty(\bs{u}_j) =I^\infty_1,
\end{align*}
so $\bs{u}$ is a minimizer. Finally, $\|\bs{u}_j-\e{3}\|_{L^2}\to\|\bs{u}-\e{3}\|_{L^2}$ and $\E^\infty(\bs{u}_j)\to \E^\infty(\bs{u})$ force $\|\Delta\bs{u}_j\|_{L^2}\to\|\Delta\bs{u}\|_{L^2}$, which together with $\Delta\bs{u}_j\rightharpoonup\Delta\bs{u}$ in $L^2(\R^3)$ gives $\Delta\bs{u}_j\to\Delta\bs{u}$ strongly in $L^2(\R^3)$. Hence $\bs{u}_j - \e 3\to\bs{u} - \e 3$ in $H^2(\R^3;\R^3)$.
\end{proof}

\appendix
\section{Rigorous numerical estimates for the unit-charge competitor}
\label{sec:appendix}
\subsection{Scale-invariant energy bound}
We provide the computational details underlying the strict bound~\eqref{eq:competitor}. The following elementary identity reduces the Laplacian of the radial ansatz to one-dimensional expressions.
\begin{lemma}
Let $F\in C^2(0,\infty)$, and let $P$ be a homogeneous polynomial of degree $\ell$ on $\R^3$. Then, for $x\neq 0$ and $r=|x|$,
\begin{align*}
    \Delta\left(
        F(r)\frac{P(x)}{r^\ell}
    \right)
    =
    (D_\ell F)(r)\frac{P(x)}{r^\ell}
    +
    F(r)\frac{\Delta P(x)}{r^\ell},
\end{align*}
where
\begin{align*}
    D_\ell F
    :=
    F''
    +
    \frac{2}{r}F'
    -
    \frac{\ell(\ell+1)}{r^2}F .
\end{align*}
\end{lemma}

\begin{proof}
Using the elementary differentiation rules for radial functions together with
\begin{align*}
    x\cdot\nabla P=\ell P,
\end{align*}
the claim follows immediately.
\end{proof}
With the explicit form
\begin{align} \label{eq:mAnsatz}
    \m(x)
    =
    \begin{pmatrix}
    2\left(
    \dfrac{x_1x_3}{r^2}\sin^2 f(r)
    +
    \dfrac{x_2}{r}\sin f(r)\cos f(r)
    \right)
    \\
    2\left(
    \dfrac{x_2x_3}{r^2}\sin^2 f(r)
    -
    \dfrac{x_1}{r}\sin f(r)\cos f(r)
    \right)
    \\
    1
    -
    2\dfrac{x_1^2+x_2^2}{r^2}\sin^2 f(r)
    \end{pmatrix},
\end{align}
together with the preceding lemma, we obtain
\begin{align*}
    \int_{\R^3}|\Delta\m|^2\,\d x
    =
    \int_0^\infty
    \Bigg[
    \frac{32\pi}{3}r^2
    \left(
    \bigl(D_1(\sin f(r)\cos f(r))\bigr)^2
    +
    \bigl(D_2(\sin^2 f(r))\bigr)^2
    \right)
    \\
    +
    \frac{256\pi}{3}
    D_2(\sin^2 f(r))\sin^2 f(r)
    +
    256\pi
    \frac{\sin^4 f(r)}{r^2}
    \Bigg]\,\d r.
\end{align*}
We also have
\begin{align*}
    \int_{\R^3} |\m - \e 3|^2 \d x  =  \frac{32\pi}{3}
    \int_0^\infty
    r^2\sin^2 f(r)
    \,\d r.
\end{align*}
For $f(r) = 2\arctan \frac{e^{-r}}{r}$, this gives
\begin{align*}
    \|\m-\e{3}\|_{L^2(\R^3)}^2 =
    \int_0^\infty\mathcal{I}_{0}(r)
    \,\d r
\end{align*}
with 
\begin{align*}
   \mathcal{I}_{0}(r)= \frac{128\pi}{3} \frac{r^4e^{2r}}{\left(1+r^2e^{2r}\right)^2}
\end{align*}
and
\begin{align*}
    \|\Delta\m\|_{L^2(\R^3)}^2 = \int_0^\infty \mathcal{I}_{2}(r)
    \,\d r
\end{align*}
with 
\begin{align*}
\mathcal{I}_{2}(r) =\frac{128\pi}{3}\frac{r^2e^{2r}}{\left(1+r^2e^{2r}\right)^4} \Big(
        &r^6e^{4r}
        -4r^4e^{4r}
        +14r^4e^{2r}
        +56r^3e^{2r}
        +4r^2e^{4r}\\
        &
        +132r^2e^{2r}
        +r^2
        +144re^{2r}
        +8r
        +96e^{2r}
        +16
    \Big).
\end{align*}
We evaluate the integrals on $[0,40]$ using arbitrary-precision ball arithmetic. This yields rigorous upper bounds for the integrals over $[0,40]$.
For the tails, we use
\begin{align*}
\mathcal I_0(r) \le \frac{128\pi}{3}e^{-2r}  \qand 
\mathcal I_2(r) \le \frac{128\pi}{3}\,472\,e^{-2r}.
\end{align*}
Consequently, for every $R\ge1$,
\begin{align*}
\int_R^\infty \mathcal I_0(r)\,\d r \le \frac{64\pi}{3}e^{-2R} \qand \int_R^\infty \mathcal I_2(r)\,\d r \le \frac{64\pi}{3}\,472\,e^{-2R}.
\end{align*}
Combining the interval enclosures on $[0,40]$ with the tail estimates
gives
\begin{align}\label{eq:chargeOneCompetitorValue} 
\frac{1}{2^4 3^{\frac38}\pi^2} \|\m-\e{3}\|_{L^2}^{\frac12} \|\Delta\m\|_{L^2}^{\frac32} < 3.636802.
\end{align}

\subsection{Quantitative existence threshold}\label{app:threshold}

\begin{lemma}
For $\m$ as in \eqref{eq:mAnsatz}, we have in spherical coordinates
\begin{align*}
    |\nabla \m|^2 = 4 f'(r)^2 \sin^2 \theta  + 4 \frac{\sin^2 f(r)}{r^2} (1+\cos^2 \theta).
\end{align*}
In particular, for $f(r) = 2 \arctan \frac{e^{-r}}{r}$
\begin{align*}
    \int_{\R^3} |\nabla \m|^2 \d x = \int_0^{\infty} \mathcal{I}_1(r) \d r
\end{align*}
with \begin{align*}
    \mathcal{I}_1 (r) = \frac{128 \pi}{3} \left( \frac{(r^2+2r+3)r^2 e^{2r}}{(1+r^2 e^{2r})^2}\right).
\end{align*}
\end{lemma}
\begin{proof}
Differentiating $\m$ in spherical coordinates yields the stated identities.
\end{proof}
As in the preceding subsection, we evaluate the integral on $[0,40]$
using arbitrary-precision ball arithmetic. For $r\ge 1$,
we have
\begin{align*}
    \mathcal I_1(r)
    \le
    \frac{128\pi}{3} 6 e^{-2r}.
\end{align*}
Consequently, for $R\ge 1$,
\begin{align*}
    \int_R^\infty\mathcal I_1(r)\,\d r
    \le
    \frac{64\pi}{3}
    6e^{-2R}.
\end{align*}
Together with the bounds established in the preceding subsection,
we obtain
\begin{align*}
    \theta_\ast
    :=
    \frac{\|\nabla\m_\ast\|_{L^2}^2}
    {\|\m_\ast-\e{3}\|_{L^2}\|\Delta\m_\ast\|_{L^2}}
    >
    0.5390.
\end{align*}
Combining these estimates with the exact scaling formula of Lemma~\ref{lemma:scaling}, we obtain
\begin{align*}
    I_1(H)
    <
    C(H)\left(2^{\frac34}+3^{\frac34}\right)
    \qquad \text{for all } H\ge 7.69.
\end{align*}

%--- bibliography ---
\bibliographystyle{plain}
\bibliography{bibliography.bib}

\end{document}